\documentclass[fleqn,12pt]{article}
\usepackage{amsmath,amssymb,amsthm,esint,dsfont,tikz}
\usepackage[english]{babel}
\usepackage[margin=3cm]{geometry}
\usepackage{csquotes}
\usepackage{hyperref}

\usepackage{color}

\newcommand{\Et}{\widetilde E}
\newcommand{\Omt}{\widetilde\Omega}
\newcommand{\nnn}{\nonumber}
\newcommand{\beqn}{\begin{equation}}
\newcommand{\eeqn}{\end{equation}}
\newcommand{\Stwo}{\mathbb{S}^2}

\newcommand{\nb}{\bar n}
\newcommand{\rt}{t}

\newcommand{\Qt}{\widetilde Q_\xi}
\newcommand{\Qb}{\overline{Q}_\xi}
\newcommand{\Om}{\Omega}

\newcommand{\abs}[1]{{\left\vert #1\right\vert}}
\newcommand{\e}{\varepsilon}
\newcommand{\R}{{\mathbb R}}

\DeclareMathOperator{\tr}{tr}

\newtheorem{theorem}{Theorem}[section]
\newtheorem{proposition}[theorem]{Proposition}
\newtheorem{lemma}[theorem]{Lemma}

\theoremstyle{definition}
\newtheorem{remark}[theorem]{Remark}

\title{Spherical particle in a nematic liquid crystal under an external field: the Saturn ring regime}
\date{}
\author{Stan Alama
\footnote{McMaster University, Hamilton, ON, Canada}
 \and Lia Bronsard
\footnotemark[\value{footnote}] 
  \and Xavier Lamy
\footnote{Institut de Math\'ematiques de Toulouse; UMR5219 - Universit\'e de Toulouse; CNRS - UPS IMT, F-31062 Toulouse Cedex 9, France}
  }

\begin{document}
\maketitle

\abstract{ We consider a nematic liquid crystal occupying the exterior region in $\R^3$ outside of a spherical particle, with radial strong anchoring.  Within the context of the Landau-de Gennes theory, we study minimizers subject to a strong external field, modelled by an additional term which favors nematic alignment parallel to the field.  When the external field is high enough we obtain a scaling law for the energy.  The energy scale corresponds to minimizers concentrating their energy in a boundary layer around the particle, with quadrupolar symmetry.  This suggests the presence of a Saturn ring defect around the particle, rather than a dipolar director field typical of a point defect.}

\section{Introduction}

In this paper we continue the study started in \cite{colloid} of a spherical colloid particle immersed in nematic liquid crystal (see also \cite{colloidphys,weakanchor}).  Motivated by the experiments described in \cite{guabbott00,loudetpoulin01} and the heuristic and numerical arguments exposed in \cite{stark02,fukudaetal04,fukudayokoyama06}, we are interested in the effect of an external magnetic or electric field on the type of defects that can be observed.

Nematic liquid crystals are typically made of elongated molecules which tend to align in a common direction. Several continuum models have been proposed to describe this alignment, including the Oseen-Frank and the Landau-de Gennes models. In the Oseen-Frank description the alignment is represented by a unit director $n\in\mathbb S^2$, and the Landau-de Gennes theory employs the so-called $Q$-tensors: traceless symmetric $3\times 3$ matrices, accounting for the alignment of the molecules through their eigenvectors and eigenvalues. With respect to directors $n\in\mathbb S^2$, the $Q$-tensors can be thought of as relaxing the uniaxial constraint
\begin{equation*}
Q=n\otimes n-\frac 13 I.
\end{equation*}
 The Landau-de Gennes energy enforces this uniaxial constraint as a small coherence length goes to zero, the limit in which one can recover the Oseen-Frank model. This convergence has recently produced a rich trove of mathematical analysis \cite{majumdarzarnescu10,canevari2d,
canevari3d,baumanparkphillips12,
golovatymontero14,singperturb}. An important feature in experiment and in the analysis is the occurrence of defects (singular structures). Compared to the director description, the additional degrees of freedom offered by $Q$-tensors allow for a much finer description of the defect cores where biaxiality might occur, and the nonlinear analysis of defect cores has recently attracted much attention \cite{INSZinstab2d,INSZstab2d,
difrattaetal16,INSZuniqhedgehog,
INSZstabhedgehog,canevari2d,biaxialescape}.

When foreign particles are immersed into nematic liquid crystal, the modifications they induce in the nematic alignment may generate additional defects, leading to many potential applications related e.g. to the detection of these foreign particles or to structure formation created by defect interactions \cite{stark01}. The mathematical analysis of such phenomena is very challenging. Here we are interested in the most fundamental situation: a single spherical particle in a sea of liquid crystal. 

We assume radial anchoring at the particle surface : the liquid crystal molecules tend to align perpendicularly to the surface. This creates a topological charge that has to be balanced by a defect so as to be compatible with a uniformly (say vertically) aligned state far away from the particle. In the absence of external field, two different types of configurations have been predicted and observed: the so-called \emph{hedgehog} and \emph{Saturn ring}. The hedgehog configuration presents one point defect above (or below) the particle. The Saturn ring configuration presents a line defect around the particle. Both configurations are axially symmetric with respect to the vertical axis, and the Saturn ring configuration enjoys the additional mirror symmetry with respect to the equatorial plane. Hedgehog configurations have been observed for large particles, and Saturn rings for smaller particles. In our previous work \cite{colloid} we provided a rigorous mathematical justification of these observations based on the Landau-de Gennes model, together with a very precise description of the Saturn ring. 

In the presence of an external field, the situation changes, as a Saturn ring defects can be observed even around large particles \cite{guabbott00}. A heuristical explanation proposed in \cite{stark02} is that the external field confines defects to a much narrower region around the particle, which is favorable to the Saturn ring type of defect. This explanation has been confirmed numerically in \cite{fukudaetal04,fukudayokoyama06} using a Landau-de Gennes model and assuming the external field to be uniform in the sample. There the presence of the external field is simply modelled by adding a symmetry-breaking term to the energy (favoring alignment along the field), multiplied by a parameter accounting for the intensity of the field. In the present paper we study this simplified model and prove that, when the applied field is high enough, minimizers should indeed correspond to Saturn ring configurations.

After adequate nondimensionalization \cite{fukudaetal04} we are left with two parameters $\xi,\eta>0$ which represent, in units of the particle radius, the coherence lengths for nematic alignment and alignment along the external field. In these units the colloid particle is represented by the closed ball of radius one $B=\lbrace \abs{\cdot}\leq 1\rbrace\subset\R^3$, so that the liquid crystal is contained in the domain  $\Omega=\R^3\setminus B$. The Landau-de Gennes energy used in \cite{fukudaetal04,fukudayokoyama06} is given by
\begin{align*}
E(Q)&=\int_{\Omega} \left( \frac 12\abs{\nabla Q}^2 +\frac{1}{\xi^2}\left[f(Q) + h^2 g(Q) \right]\right)\\
&=\int_{\Omega} \left( \frac 12\abs{\nabla Q}^2 +\frac{1}{\xi^2}f(Q) + \frac{1}{\eta^2} g(Q) \right),\quad\eta=\frac\xi h.
\end{align*}
The map $Q$ takes values into the space $\mathcal S_0$ of $3\times 3$ symmetric matrices with zero traces and describes nematic alignment.
 The nematic potential $f(Q)=-\frac 12\abs{Q}^2-\tr(Q^3)+\frac 34 \abs{Q}^4+cste$ satisfies
\begin{equation*}
f(Q)\geq 0\text{ with equality iff }Q=n\otimes n -\frac 13 I\text{ for some }n\in\mathbb S^2.
\end{equation*}
The symmetry-breaking potential $g(Q)$ is given by
\begin{equation*}
g(Q)=\sqrt{\frac 23} - \frac{Q_{33}}{\abs{Q}}.
\end{equation*}
It breaks symmetry in the sense that the rotations $R\in SO(3)$ which satisfy $g({}^t R Q R)=g(Q)$ for all $Q\in\mathcal S_0$ must have $e_3$ as an eigenvector, while $f({}^tR Q R)=f(Q)$ for all $R\in SO(3)$ and $Q\in \mathcal S_0$. Its specific form is chosen so that
\begin{equation*}
g(Q)=c(1-n_3^2)\quad\text{for }Q=n\otimes n-\frac 13 I,
\end{equation*}
and $g(Q)$ is invariant under multiplication of $Q$ by a positive constant \cite{fukudaetal04}.
This potential satisfies
\begin{equation*}
g(Q)\geq 0\text{ with equality iff }Q=\lambda\left(\mathbf e_3\otimes \mathbf e_3 -\frac 13 \right)I\text{ for some }\lambda>0.
\end{equation*}
Hence for $h>0$ the full potential
$f(Q)+hg(Q)$
is minimized exactly at $Q=Q_\infty$, where
\begin{equation*}
Q_\infty=\mathbf e_3 \otimes \mathbf e_3 -\frac 13 I.
\end{equation*}
Moreover it is easily checked that
\begin{equation}\label{eq:coercpot}
f(Q)+hg(Q)\geq C(h)\abs{Q-Q_\infty}^2,
\end{equation}
for some constant $C(h)>0$. This ensures that the energy is coercive on the affine space $Q_\infty + H^1(\Omega;\mathcal S_0)$.
The anchoring at the particle surface is assumed to be radial:
\begin{equation}\label{eq:Qb}
Q=Q_b:=\mathbf e_r\otimes \mathbf e_r -\frac 13 I\qquad\text{on }\partial \Omega,\qquad \mathbf e_r=\frac{x}{\abs{x}}.
\end{equation}
Denoting  by $\mathcal H$ the space
\begin{equation}\label{eq:H}
\mathcal H = \left\lbrace Q\in Q_\infty + H^1(\Omega;\mathcal S_0)\colon Q=Q_b\text{ on }\partial\Omega\right\rbrace,
\end{equation}
the coercivity of the energy ensures existence of a minimizer in $\mathcal H$ for any $\xi,\eta>0$.

In \cite[\S~3.1]{fukudayokoyama06},  heuristic arguments are used to estimate the behavior of the different terms in the energy for a `hedgehog' configuration and for  a `Saturn ring' configuration. Decomposing the energy as $E=E_{nem}+E_{mag}$, where
\begin{gather*}
E_{nem}=\int_{\R^3\setminus B} \left( \frac 12\abs{\nabla Q}^2 +\frac{1}{\xi^2}f(Q)  \right),\qquad
E_{mag}=\frac{1}{\eta^2}\int_{\R^3\setminus B} g(Q),
\end{gather*}
they conjecture the asymptotics
\begin{align*}
\text{hedgehog:}&\quad E_{nem} \approx 1,\; E_{mag}\approx \frac{1}{\eta},\\
\text{Saturn ring:}&\quad E_{nem} \approx \abs{\ln\xi},\; E_{mag}\approx 1.
\end{align*}
In fact we believe that for the Saturn ring the magnetic part of the energy should also be of order $1/\eta$ (with a smaller constant though), but this does not affect the conclusion that there should be a critical value
\begin{equation*}
\eta_c\approx\frac{1}{\abs{\ln\xi}},
\end{equation*}
with the following properties. If $\eta<\eta_c$ (high applied field) then the Saturn ring configuration has lowest energy, and if $\eta>\eta_c$ (low applied field) then the hedgehog configuration has lowest energy.
In  \cite{fukudaetal04,fukudayokoyama06} this conjecture is checked numerically, for $\xi=4\times 10^{-3}$ and $\eta$ around $10^{-1}$ (so that $h$ lies between $10^{-2}$ and $10^{-1}$) in \cite{fukudaetal04}, and  $\xi=\overline R_0^{-1}$ between $10^{-3}$ and $10^{-2}$ and $h$ of the same order in \cite{fukudayokoyama06}.

Our aim in this work is to justify rigorously the fact that the Saturn ring configuration is minimizing for high fields, i.e. $\eta\ll 1/\abs{\ln\xi}$. We will tackle the regime $\eta\gg 1/\abs{\ln\xi}$ in a forthcoming work \cite{lowfield}. Since physically relevant values of $\eta,\xi$ satisfy $\xi\lesssim \eta\ll 1$, we consider the limit $\xi\to 0$ and assume that 
\begin{equation*}
\eta=\eta(\xi)\longrightarrow 0\qquad\text{as }\xi\to 0.
\end{equation*}
With this convention the energy functional depends only on the small parameter $\xi$ and we write
\begin{equation*}
E_\xi(Q;U)=\int_U\left[\frac 12 \abs{\nabla Q}^2 +  \frac{1}{\xi^2}f(Q) + \frac{1}{\eta^2}g(Q)\right],
\end{equation*}
for any measurable set $U\subset\Omega=\mathbb R^3\setminus B$ and $Q\in Q_\infty + H^1(\Omega;\mathcal S_0)$.

One way to characterize a Saturn ring configuration versus a hedgehog configuration is its mirror symmetry: a Saturn ring configuration is symmetric with respect to the equatorial plane $\lbrace x_3=0\rbrace$, while a hedgehog configuration is not; that is,
\begin{equation*}
E_\xi(Q^{Saturn};\Omega_+)=E_\xi(Q^{Saturn};\Omega_-),\qquad\Omega_\pm = \Omega\cap\lbrace \pm x_3 >0\rbrace,
\end{equation*}
while for a hedgehog configuration these energies are different. Our first main result shows that a minimizing configuration has to satisfy this symmetry asymptotically if $\xi\lesssim \eta\ll\abs{\ln\xi}^{-1}$.

\begin{theorem}\label{t:sym}
If $Q_\xi$ minimizes $E_\xi$ with boundary conditions \eqref{eq:Qb} and 
\begin{equation*}
\eta(\xi)\abs{\ln\xi}\to 0,\qquad \frac{\eta(\xi)}{\xi}\to\lambda\in (0,\infty],
\end{equation*}
then
\begin{equation*}
E_\xi(Q_\xi;\Omega_+)\sim E_\xi(Q_\xi;\Omega_-)\qquad\text{as }\xi\to 0.
\end{equation*}
\end{theorem}

Theorem~\ref{t:sym} is a consequence of the more precise asymptotics we obtain for the energy of a minimizer. The potential
\begin{equation*}
\frac{1}{\xi^2}f(Q)+\frac{1}{\eta^2}g(Q) = \frac{1}{\eta^2}\left(\frac{\eta^2}{\xi^2}f(Q)+g(Q)\right),
\end{equation*}
is minimized at $Q=Q_\infty$. As $\eta\to 0$ this forces a minimizing configuration to be very close to $Q_\infty$. The boundary data $Q_b$ satisfies $f(Q_b)\equiv 0$ but not $g(Q_b)\equiv 0$. Not surprisingly, deformations concentrate in a boundary layer of size $\eta$, where a one-dimensional transition takes place according to the energy
\begin{equation}\label{Flambda}
F_\lambda(Q)=\int_1^\infty \left[\frac 12  \abs{\frac{d Q}{d r}}^2 +\lambda^2 f(Q)+g(Q)\right] \, dr,
\end{equation}
defined for $Q\in Q_\infty + H^1((1,\infty);\mathcal S_0)$.
For $\lambda=\infty$ this formula should be understood as
\begin{equation*}
F_\infty(Q)=
\left\lbrace
\begin{aligned}
&\int_1^\infty \left[\frac 12  \abs{\frac{d Q}{d r}}^2 + g(Q)\right] \, dr &&\text{ if }f(Q)=0\text{ a.e.},\\
&+\infty && \text{ otherwise.}
\end{aligned}
\right.
\end{equation*}
In other words, $F_\infty$ is finite for maps $Q\in Q_\infty + H^1((1,\infty);\mathcal S_0)$ which satisfy $Q=n\otimes n -I/3$ for some measurable map $n\colon (0,\infty)\to\mathbb S^2$.

Obviously the cases $\lambda\in (0,\infty)$ and $\lambda=\infty$ are quite different and they require separate treatments, but in both cases we obtain for the energy of a minimizer $Q_\xi$ the asymptotics 
\begin{equation*}
E_\xi(Q_\xi;\Omega)=\frac{1}{\eta}\int_{\mathbb S^2}D_\lambda(Q_b(\omega))\,d\mathcal H^2(\omega) + o\left(\frac 1\eta\right)\qquad\text{as }\xi\to 0,
\end{equation*}
where
\begin{equation}\label{eq:Dlambda}
D_\lambda(Q_0)=\min\left\lbrace F_\lambda(Q)\colon Q\in Q_\infty +H^1((1,\infty);\mathcal S_0),\,Q(1)=Q_0\right\rbrace.
\end{equation}
The existence of a minimizer of $F_\lambda$ which attains $D_\lambda(Q_b(\omega))$ for any $\omega\in\Stwo$ follows from the direct method.  In section~\ref{ss:upII} we will exploit the observation of Sternberg \cite{St91} that the heteroclinic connections which minimize $F_\lambda$ represent geodesics for a degenerate metric.
In the case $\lambda=\infty$ this enables us to obtain an exact value for the limiting energy,
\beqn\label{liminfty}  D_\infty(Q_b(\theta,\varphi))=
     \kappa(1-|\cos\theta|), \quad \text{and}
\quad \lim_{\xi\to 0} \eta\, E_\xi(Q_\xi;\Om) = 2\pi\kappa,
\eeqn
where $\kappa:=\sqrt[4]{24}$. (See Lemma~\ref{L:geod}.)

More specifically, we obtain local asymptotics in angular subdomains of $\Omega$. For $U\subset\mathbb S^2$ we denote by $\mathcal C(U)$ the cone
\begin{equation*}
\mathcal C(U)=\left\lbrace t\omega\colon t>1,\,\omega\in U\right\rbrace,
\end{equation*}
and prove
\begin{theorem}\label{t:asympt}
If $Q_\xi$ minimizes $E_\xi$ with boundary conditions \eqref{eq:Qb} and 
\begin{equation*}
\eta(\xi)\abs{\ln\xi}\to 0,\qquad \frac{\eta(\xi)}{\xi}\to\lambda\in (0,\infty],
\end{equation*}
then for any measurable set $U\subset\mathbb S^2$ it holds
\begin{equation*}
E_\xi(Q_\xi;\Omega\cap\mathcal C(U))=\frac 1\eta \int_{U} D_\lambda(Q_b(\omega))\,d\mathcal H^2(\omega)+o\left(\frac 1\eta\right)\qquad\text{as }\xi\to 0.
\end{equation*}
\end{theorem}
Theorem~\ref{t:sym} follows trivially from Theorem~\ref{t:asympt} by applying the latter to $U=(\mathbb S^2)_\pm = \mathbb S^2\cap\lbrace \pm x_3>0\rbrace$, since $Q_b$ is symmetric under reflection with respect to the equatorial plane.

The lower bound in Theorem~\ref{t:asympt} follows from an elementary rescaling and the properties of $\lambda\mapsto D_\lambda$. To obtain an upper bound matching this lower bound, a first approach  would be to define a trial map $Q$ on every radial direction by an appropriate rescaling of a minimizer of $F_\lambda$, i.e. set
\begin{equation*}
Q(r\omega)=Q_\omega\left(1+\frac{r-1}{\eta}\right)\qquad\text{for }(r,\omega)\in (1,\infty)\times\mathbb S^2,
\end{equation*}
with $Q_\omega$ minimizing $F_\lambda$ under the constraint $Q_\omega(1)=Q_b(\omega)$. The problem with this approach is that it may not be possible to control the derivatives of such $Q$ with respect to angular variable $\omega$. We overcome this difficulty by using different arguments in the cases
$\lambda\in (0,\infty)$ and $\lambda=\infty$. 

For $\lambda\in (0,\infty)$ we  take advantage of the fact that, although the regularity of $\omega\mapsto Q_\omega$ is not understood, the map $\omega\mapsto F_\lambda(Q_\omega)=D_\lambda(Q_b(\omega))$ is easily seen to be continuous, hence Riemann integrable. Thus we obtain a trial map by smoothly interpolating between $Q_{\omega_i}(r)$ for a discrete set  $\lbrace \omega_i\rbrace$. The cost of this weak approach is that we cannot hope to obtain a more precise remainder than $o(1/\eta)$.

For $\lambda=\infty$ the map $Q_\omega$ takes  the form $n\otimes n-I/3$ and this restriction allows to specify its dependence on $\omega$. However the topological constraint enforced by the boundary conditions prevents it to be smooth : there is a jump as $\omega$ crosses the equatorial plane $\lbrace x_3=0\rbrace$. We modify the trial map near this plane by including a Saturn ring defect which rectifies the topological charge. With this approach the remainder in the upper bound is actually of the order $O(\abs{\ln\xi})$.

Finally, it is natural and tempting to make a direct comparison between the symmetric minimizer (which we expect to represent the Saturn ring) and its usual competitor, the dipolar hedgehog. The difficulty is that we do not know if there exists such a solution, nor how to impose constraints under which there would be a minimizer of this form. However, we can restrict our attention to uniaxial tensors with oriented director fields, $Q=n\otimes n -{\frac13} I$, $n\in \mathcal N$, where 
$$\mathcal N:= \left\{ n\in H^1_{loc}(\Omega;\Stwo): \  n|_{\partial \Omega} =\mathbf e_r, \ \int_\Om (n_1^2+n_2^2) dx<\infty.\right\}$$
 Within this  orientable setting, the Saturn ring line defect is not admissible anymore since it carries a half-integer degree \cite{ballzarnescu11}.  We show that orientable configurations have much larger energy at leading order:

\begin{proposition}\label{p:enercomp} Let $Q_\xi$ minimize $E_\xi$ with boundary conditions \eqref{eq:Qb} and $\eta=\eta(\xi)$ with
\begin{equation*}
\eta\abs{\ln\xi}\to 0,\qquad \frac{\eta}{\xi}\to\lambda\in (0,\infty].
\end{equation*}
Then, for any $Q=n\otimes n -{\frac13} I$ with $n\in \mathcal N$, and any $\xi>0$, we have 
$$  \eta E_\xi(Q)\ge 8\pi\kappa \ge 4 \lim_{\xi\to 0
}\left(\eta E_\xi(Q_\xi)\right).  $$
\end{proposition}

\medskip

The paper is organized as follows. In \S~\ref{s:low} we prove the lower bound. In \S~\ref{s:up} we concentrate on the upper bound, considering the case $\lambda\in (0,\infty)$ in \S~\ref{ss:upI} and $\lambda=\infty$ in \S~\ref{ss:upII}. We conclude with the short proofs of Theorem~\ref{t:asympt} and Proposition~\ref{p:enercomp} in \S~\ref{s:proof}.

\paragraph{Acknowledgements:} We thank E.C.~Gartland for useful discussions on nondimensionalization.  SA and LB were supported by NSERC (Canada) Discovery Grants.

\section{Lower bound}
\label{s:low}

In this section we prove the
\begin{proposition}\label{p:low}
If $Q_\xi$ minimizes $E_\xi$ with boundary conditions \eqref{eq:Qb} and 
\begin{equation*}
 \eta=\eta(\xi)\to 0, \quad \frac{\eta}{\xi}\to\lambda\in (0,\infty], 
\end{equation*}
then for any measurable set $U\subset\mathbb S^2$ it holds
\begin{equation*}
\liminf_{\xi\to 0} \eta E_\xi(Q_\xi;\Omega\cap\mathcal C(U))\geq \int_{U} D_\lambda(Q_b(\omega))\,d\mathcal H^2(\omega).
\end{equation*}
\end{proposition}
\begin{proof}
We use spherical coordinates $x=r\omega$, $(r,\omega)\in (1,\infty)\times\mathbb S^2$.
 Setting $r=1+\eta(\tilde r -1)$ and
 \begin{equation*}
 \widetilde Q(\tilde r,\omega)=Q_\xi(1+\eta(\tilde r -1),\omega),
 \end{equation*}
  we have
\begin{align*}
\eta E_\xi(Q_\xi;\mathcal C(U))&
=\int_U \int_1^\infty  \Bigg[\frac 12 \abs{\frac{\partial\widetilde Q}{\partial\tilde r}}^2
+\frac{\eta^2}{r^2}\abs{\nabla_\omega\widetilde Q}^2
+ \frac{\eta^2}{\xi^2}f(\widetilde Q)+g(\widetilde Q)
\Bigg]\,r^2 d\tilde r \, d\mathcal H^2(\omega)\\
&\geq \int_U \int_1^\infty  \Bigg[\frac 12 \abs{\frac{\partial\widetilde Q}{\partial\tilde r}}^2 
+  \frac{\eta^2}{\xi^2}f(\widetilde Q)+g(\widetilde Q)
\Bigg]\, d\tilde r \, d\mathcal H^2(\omega)\\
&\geq \int_U D_{\frac\eta\xi}(Q_b(\omega))\,d\mathcal H^2(\omega),
\end{align*}
using \eqref{Flambda} and \eqref{eq:Dlambda} for the last inequality.
We conclude using the fact (see Lemma~\ref{l:Dlambda} below) that
\begin{equation*}
D_\lambda(Q_b(\omega))=\lim_{\mu\to\lambda} D_\mu(Q_b(\omega))\qquad\forall\omega\in\mathbb S^2,
\end{equation*}
and Fatou's lemma.
\end{proof}

\begin{lemma}\label{l:Dlambda}
For any $Q_0\in\mathcal S_0$ and $\lambda\in (0,\infty]$ we have
\begin{equation*}
D_\lambda(Q_0)=\lim_{\mu\to\lambda}D_\mu(Q_0).
\end{equation*}
\end{lemma}
\begin{proof} The arguments are standard, we only sketch them here.

In the case $\lambda=\infty$ and $f(Q_0)> 0$ we have $D_\infty(Q_0)=\infty$. Then we also have $D_\mu(Q_0)\to\infty$ as $\mu\to\infty$. Otherwise there would exist a sequence $\mu_k\to\infty$ and maps $Q^k$ with $Q^k(1)=Q_0$ such that $F_{\mu_k}(Q^k)\leq C$, and therefore up to a subsequence $Q^k$ converges weakly in $H^1((1,\infty);\mathcal S_0)$ to a map $Q^*$ with $Q^*(1)=Q_0$.  However, the bound $\mu_k^2\int f(Q^k)\leq C$ implies that $f(Q^*)=0$ a.e., contradicting $f(Q_0)>0$.

Hence we may assume that $D_\lambda(Q_0)<\infty$, and pick a minimizer $Q^\lambda$ of $F_\lambda$ with $Q^\lambda(1)=Q_0$. Fix a sequence $\mu_k\to\lambda$ and minimizers $Q^k$ of $F_{\mu_k}$ with $Q^k(1)=Q_0$. Then we have the bound
\begin{equation*}
\limsup_{k\to\infty} F_{\mu_k}(Q^k)\leq \limsup_{k\to\infty} F_{\mu_k}(Q^\lambda) =F_\lambda(Q^\lambda)=D_\lambda(Q_0),
\end{equation*}
and therefore, up to a subsequence, $Q^k$ converges weakly in $H^1((1,\infty);\mathcal S_0)$ towards a map $Q^*$. The weak lower semi-continuity of $\int\abs{dQ/dr}^2$ and Fatou's lemma then imply
\begin{equation*}
F_\lambda(Q^*)\leq\liminf_{k\to\infty} F_{\mu_k}(Q^k),
\end{equation*}
so that combining the above we have
\begin{equation*}
D_\lambda(Q_0)\leq F_\lambda(Q^*)
\leq\liminf_{k\to\infty} F_{\mu_k}(Q^k) 
\leq \limsup_{k\to\infty} F_{\mu_k}(Q^k)
\leq D_\lambda(Q_0),
\end{equation*}
and deduce $\lim F_{\mu_k}(Q^k)=\lim D_{\mu_k}(Q_0)=D_\lambda(Q_0)$.
\end{proof}

\section{Upper bound}
\label{s:up}

\subsection{The case $\lambda\in (0,\infty)$}
\label{ss:upI}

In this section we assume that
\begin{equation*}
\frac\eta\xi \longrightarrow \lambda\in (0,\infty)\qquad\text{as }\xi\to 0,
\end{equation*}
and show that
\begin{equation*}
\min_{\mathcal H} E_\xi \leq \frac{1}{\eta}\int_{\mathbb S^2}D_\lambda(Q_b(\omega))\, d\mathcal H^2(\omega) + o\left(\frac 1\eta\right),
\end{equation*}
as $\xi\to 0$. This is obtained by constructing an admissible comparison map. This comparison map depends on two parameters $\e,h>0$, the use of which will become clear in the course of the proof.

\begin{proposition}
\label{p:upI}
For any $\varepsilon,h,\xi>0$ there exists a map $Q_\xi^{h,\e}$ such that
\begin{equation*}
\limsup_{\xi\to 0} \eta  E_\xi(Q^{h,\e}_\xi) \leq  \int_{\mathbb S^2}D_\lambda(Q_b(\omega))\, d\mathcal H^2(\omega) + \sigma(h,\e),
\end{equation*}
where $\lim_{h\to 0}(\lim_{\e\to 0}\sigma(h,\e)) =0$.
\end{proposition}
\begin{proof}
We construct an axially symmetric  map $Q_\xi^{h,\e}$ of the form
\begin{equation*}
Q_\xi^{h,\e}(r,\theta,\varphi)
= {}^tR_\varphi \widetilde Q^{h,\e}\left(1+\frac{r-1}{\eta},\theta\right) R_\varphi,
\end{equation*}
where $\widetilde Q^{h,\e}(\tilde r, \theta)$ is a smooth map to be determined later, and $R_\varphi$ is the rotation of angle $\varphi$ and axis $\mathbf e_3$.
Dropping the  exponents $h,\e$ (as we will do when there is no confusion) it holds
\begin{align*}
\abs{\nabla Q_\xi}^2 &=
\frac{1}{\eta^2}\abs{\frac{\partial\widetilde Q}{\partial\tilde r}}^2 + \frac{1}{r^2}\abs{\frac{\partial\widetilde Q}{\partial\theta}}^2+\frac{1}{r^2\sin^2\theta}\Xi[\widetilde Q],\\
\text{where }\Xi[\widetilde Q]&= \abs{ \partial_\varphi[ {}^t R_\varphi \widetilde Q R_\varphi]}^2.
\end{align*}
The function $\Xi$ is a nonnegative quadratic form with bounded coefficients depending smoothly on $\varphi$. Since $Q_\infty$ commutes with $R_\varphi$, $\Xi[\widetilde Q]$ vanishes at $\widetilde Q=Q_\infty$  and  satisfies
\begin{equation*}
\Xi[\widetilde Q]\leq C \abs{\widetilde Q-Q_\infty}^2,
\end{equation*}
for some absolute constant $C>0$.
Integrating in $\Omega$ and changing variables according to $r-1=\eta (\tilde r-1)$ we find
\begin{align*}
\eta \int_\Omega\abs{\nabla Q_\xi}^2 &
 \leq 2\pi \int_0^\pi \int_1^\infty \abs{\frac{\partial \widetilde Q}{\partial \tilde r}}^2 d\tilde r\, \sin\theta d\theta +2\pi \eta\, \mathcal R_\xi(\widetilde Q),\\
\mathcal R_\xi(\widetilde Q)& =  \int_0^\pi \int_1^\infty (2(\tilde r -1) + \eta (\tilde r-1)^2)\abs{\frac{\partial \widetilde Q}{\partial \tilde r}}^2\sin\theta d\theta \\
&\quad + \eta\int_0^\pi \int_1^\infty \abs{\frac{\partial \widetilde Q}{\partial \theta}}^2 d\tilde r\, \sin\theta d\theta + \eta\int_0^\pi \int_1^\infty \frac{C}{\sin\theta}\abs{\widetilde Q-Q_\infty}^2 d\tilde r\, d\theta.
\end{align*}
For any fixed $h,\e>0$ we will have $\sup_\xi \mathcal R_\xi(\widetilde Q^{h,\e})<\infty$ provided 
\begin{equation}\label{eq:tildeQcompact}
\widetilde Q^{h,\e}-Q_\infty\equiv 0\text{ outside a compact, and near }\theta=0\text{ and }\theta=\pi.
\end{equation}
Moreover we have
\begin{align*}
\eta\int_\Omega &\left(\frac{1}{\xi^2}f(Q_\xi)+\frac{1}{\eta^2}g(Q_\xi)\right) \\
& =
2\pi \int_0^\pi \int_1^\infty\left( \lambda^2 f(\widetilde Q)+g(\widetilde Q) \right) d\tilde r\, \sin\theta d\theta \\
&\quad +2\pi \eta \, \int_0^\pi \int_1^\infty
(2(\tilde r -1)+\eta(\tilde r-1)^2))\left(\frac{\eta^2}{\xi^2}f(\widetilde Q)+g(\widetilde Q)\right)
d\tilde r\, \sin\theta d\theta\\
&\quad + 2\pi\left(\frac{\eta^2}{\xi^2}-\lambda^2\right)\int_0^\pi \int_1^\infty f(\widetilde Q)\, d\tilde r\,\sin\theta d\theta.
\end{align*}
Since $f(Q_\infty)=g(Q_\infty)=0$, if \eqref{eq:tildeQcompact} is satisfied for all $h,\e>0$, we deduce, gathering the above,
\begin{align}\label{eq:EQ}
\limsup_{\xi\to 0} \eta E_\xi(Q^{h,\e}_\xi)&\leq 2\pi \int_0^\pi F_\lambda\left(\widetilde Q^{h,\e}(\cdot,\theta)\right) \sin\theta d\theta,
\end{align}
where $F_\lambda$ was defined in \eqref{Flambda}.
Recall that
\begin{equation*}
D_\lambda(Q_0)=\inf\left\lbrace F_\lambda(\widetilde Q)\colon \widetilde Q(1)=Q_0\right\rbrace.
\end{equation*}
The functional $F_\lambda$ is invariant under pointwise conjugation by $R_\varphi$ for any angle $\varphi\in\R$, and therefore
\begin{equation*}
D_\lambda({}^t R_\varphi Q_0 R_\varphi)=D_\lambda(Q_0).
\end{equation*}
Since $Q_b(\omega)=Q_b(\theta,\varphi)$ is axially symmetric, in other words
\begin{equation*}
Q_b(\theta,\varphi)={}^t R_\varphi Q_b(\theta,0)R_\varphi,
\end{equation*}
we deduce that $D_\lambda(Q_b(\theta,\varphi))$ does not depend on the azimuthal angle $\varphi$, and
\begin{equation*}
\int_{\mathbb S^2} D_\lambda(Q_b(\omega)) d\mathcal H^2(\omega) = 2\pi\int_0^\pi D_\lambda(Q_b(\theta,0))\, \sin\theta d\theta.
\end{equation*}
Combining this with \eqref{eq:EQ}, in order to prove Proposition~\ref{p:upI} it suffices to construct for all $h,\e>0$ a smooth map $\widetilde Q^{h,\e}(\tilde r,\theta)$ which satisfies \eqref{eq:tildeQcompact} and
\begin{equation}\label{eq:upQtilde}
\limsup_{h\to 0}\left[\limsup_{\e\to 0}\int_0^\pi F_\lambda\left(\widetilde Q^{h,\e}(\cdot,\theta)\right) \sin\theta d\theta \right]\leq \int_0^\pi D_\lambda(Q_b(\theta,0)) \sin\theta d\theta.
\end{equation}
In principle one would like to choose $\widetilde Q(\cdot,\theta)$ minimizing $F_\lambda$ with respect to the boundary condition $\widetilde Q(1,\theta)=Q_b(\theta,0)$. But it is not obvious that such a map $\widetilde Q$ would be (even weakly) differentiable in  $\theta$. However we can make use of the continuity of $\theta\mapsto D_\lambda(Q_b(\theta,0))$ to bypass this issue, at the price of introducing the extra parameters $h,\e>0$. 

Thanks to Lemma~\ref{l:Dlip} below, the function $\theta\mapsto  D_\lambda(Q_b(\theta,0))\sin\theta$ is continuous on $[0,\pi]$. In particular it is Riemann integrable, and there exists a family of partitions
\begin{equation*}
0=\theta_1^h<\theta_2^h<\cdots<\theta_{I_h}^h =\pi,\quad\sup_i\abs{\theta^h_{i+1}-\theta^h_i}\leq h,
\end{equation*}
such that
\begin{equation*}
\lim_{h\to 0} \sum_{i}(\theta^h_{i+1}-\theta^h_i)D_\lambda(Q_b(\theta_i^h,0))\sin\theta_i^h = \int_0^\pi D_\lambda(Q_b(\theta,0))\sin\theta\, d\theta.
\end{equation*}
For any $i\in \lbrace 1,\ldots, I_h-1\rbrace$ there exists a map $\widetilde Q_i^h(\tilde r)$ such that $\widetilde Q_i^h-Q_\infty\in C_c^\infty([1,\infty))$ and
\begin{equation*}
F_\lambda(\widetilde Q_i^h)\leq D_\lambda(Q_b(\theta_i^h,0)) + h.
\end{equation*} 
Then, defining
\begin{equation*}
\widetilde Q^h(r,\theta)=\begin{cases}
Q_\infty &\text{ if }\theta\in [0,\theta_2^h)\cup [\theta_{I_h-1}^h,\pi),\\
\widetilde Q_i^h(r)&\text{ if }\theta\in [\theta_i^h,\theta_{i+1}^h),\; 2\leq i\leq I_h-2,
\end{cases}
\end{equation*}
we obtain
\begin{equation}\label{eq:upQh}
\limsup_{h\to 0} \int_0^{\pi} F_\lambda(\widetilde Q^{h}(\cdot,\theta))\sin\theta \, d\theta \leq \int_0^\pi D_\lambda(Q_b(\theta,0))\sin\theta\, d\theta.
\end{equation}
Eventually we define $\widetilde Q^{h,\e}$ by smoothing $\widetilde Q^h$ in $\theta$, i.e.
\begin{equation*}
\widetilde Q^{h,\e}=Q^h *_\theta\varphi_\e,
\end{equation*}
for some smooth kernel $\varphi_\e(\theta)=\e^{-1}\varphi(\theta/\e)$. Such map $\widetilde Q^{h,\e}$  satisfies \eqref{eq:tildeQcompact}, and 
$$\widetilde Q^{h,\e}\longrightarrow \widetilde Q^{h},\qquad {\partial \widetilde Q^{h,\e}\over \partial r}\longrightarrow {\partial \widetilde Q^{h}\over \partial r}\quad a.e.$$
By dominated convergence we thus have
\begin{equation*}
\lim_{\e\to 0} \int_0^{\pi} F_\lambda(\widetilde Q^{h,\e}(\cdot,\theta))\sin\theta \, d\theta = \int_0^{\pi} F_\lambda(\widetilde Q^{h}(\cdot,\theta))\sin\theta \, d\theta.
\end{equation*}
Combining this with \eqref{eq:upQh} we obtain \eqref{eq:upQtilde}, thus completing the proof.
\end{proof}

\begin{lemma}\label{l:Dlip}
The map $Q_0\mapsto D_\lambda(Q_0)$ is locally Lipschitz.
\end{lemma}
\begin{proof}
Let $Q_0^1,Q_0^2\in\mathcal S_0$ be such that $\abs{Q_0^1},\abs{Q_0^2}\leq M$. Let $\widetilde Q^1$ be such that
\begin{equation*}
D_\lambda(Q_0^1)=F_\lambda(\widetilde Q^1),\qquad \widetilde Q^1(1)=Q_0^1.
\end{equation*}
Let $\delta>0$ and define
\begin{equation*}
\widetilde Q^2(r)=\begin{cases}
Q_0^2+\frac{ r-1}{\delta}(Q_0^1-Q_0^2) &\text{ for }1<r<1+\delta,\\
\widetilde Q^1(r-\delta) &\text{ for }r>1+\delta.
\end{cases}
\end{equation*}
Then 
\begin{align*}
D_\lambda(Q_0^2)-D(Q_0^1)&\leq F_\lambda(\widetilde Q^2) -F_\lambda(\widetilde Q^1)\\
& =\frac 12 \int_1^{1+\delta}
\left[\frac{\abs{Q_0^1-Q_0^2}^2}{\delta^2}+\lambda^2 f(\widetilde Q^2)+ g(\widetilde Q^2)\right] dr \\
&\leq \frac{\abs{Q_0^1-Q_0^2}^2}{2\delta}+ \frac C 2 \delta,
\end{align*}
for $C=\sup_{\abs{Q}\leq M}(\lambda^2 f+g)$. Choosing $\delta=C^{-1/2}\abs{Q_0^1-Q_0^2}$ yields
\begin{equation*}
D_\lambda(Q_0^2)-D(Q_0^1) \leq \frac { C^{1/2}}{2}\abs{Q_0^1-Q_0^2},
\end{equation*}
thus proving the local Lipschitz continuity of $D_\lambda$.
\end{proof}

\subsection{The case $\lambda=\infty$}
\label{ss:upII}

We next consider a compementary regime to the one considered above, with $\eta=\eta(\xi)$ such that
\begin{equation*}
\e:=\frac\xi\eta \longrightarrow 0\quad\text{and}\quad
\eta|\ln\xi|\longrightarrow 0\qquad\text{as }\xi\to 0,
\end{equation*}
that is, the characteristic length scale determined by the field is much larger than the length scale determined by elastic response in the nematic.
Again, we derive an upper bound on the energy by constructing an appropriate test map, whose structure suggests the anticipated form of the minimizers.  We show:

\begin{proposition}
\label{p:upII}
There exists a map $Q_\xi$ such that
\begin{equation*}
  E(Q_\xi) \leq  {1\over \eta} \int_{\mathbb S^2}D_\infty(Q_b(\omega))\, d\mathcal H^2(\omega) + \frac 23 \pi^2 |\ln \e| + O(1) .
\end{equation*}
\end{proposition}

Before proving the proposition we require some further information about the minimizing geodesic of the problem $D_\infty$.  Recall that this minimization is taken over uniaxial tensors, and thus reduces to a problem for unit vector fields $n\in\Stwo$.  We note that for $Q=n\otimes n -\frac13 I$, the magnetic energy density is expressed as
$$  g(Q)= \sqrt \frac32(1-n_3^2)=: g(n),  $$
with a slight abuse of notation.
This is both a major simplification and a minor complication:  whereas the potential vanishes for exactly one uniaxial tensor $Q_\infty$, it vanishes for two antipodal directors $n=\pm \mathbf e_3$.  We denote by $\omega(\theta,\varphi)$ the point on $\Stwo$ with angular coordinates $(\theta,\varphi)\in [0,\pi]\times [0,2\pi)$.
Define
$$  F_\infty(n):= \int_0^\infty \left( |\dot n|^2 + g(n) \right) dt,   $$
with $n\in H^1_{loc}( [0,\infty); \Stwo)$ with $n(0)=\omega(\theta,\varphi)$.
Finiteness of the energy enforces the condition $n(t)\to\pm \mathbf e_3$ as $t\to\infty$, but the choice of terminal point will depend on the initial value $n(0)\in\Stwo$.
Let 
$$d_\infty^\pm(\omega):= \inf_{n(0)=\omega \atop n(\infty)=\pm \mathbf e_3} F_\infty (n).
$$
Then, since $n(0)=\omega$ is chosen such that $Q_b(\omega(\theta,\varphi))=n(0)\otimes n(0)-\frac13 I$, we have
$$  D_\infty( Q_b(\omega) ) = \min\left\{
d_\infty^+(\omega), d_\infty^-(\omega)
\right\}.  $$
By symmetry it is enough to consider the case where the target point is $+\mathbf e_3$.  We have the following characterization of the minimizers:

\begin{lemma}\label{L:geod} For any $\omega(\theta,\varphi)\in\Stwo$ there exists a minimizer $n=n(t,\theta,\varphi)$ of $d_\infty^+(\omega)$, with
$$  F_\infty(n) =  \kappa(1-\cos\theta), \qquad \kappa=\sqrt[4]{24}.  $$
  The minimizer is $C^1$ smooth and equivariant, that is $n(t,\theta,\varphi)=R_\varphi n(t,\theta,0)$ for all $\varphi$.  Moreover, we have 
\beqn\label{eq:decay}
|n(t,\theta,\varphi)-\mathbf e_3|, |\dot n(t,\theta,\varphi)|^2, 
\left| {\partial n\over\partial\theta} \right|^2 \le Ce^{-\kappa t}
\eeqn
for constant $C>0$, uniformly in $\theta,\varphi$.
\end{lemma}

\begin{proof}  The existence of a minimizer for each fixed $(\theta,\varphi)$ follows from \cite{St91}; the other statements are special to our case.  First, we note that for any rotation $R_\varphi$, $F_\infty(R_\varphi n)=F_\infty(n)$, and thus it is sufficient to consider the case $\varphi=0$.  We claim that given any admissible $n(t)=(n_1,n_2,n_3)$, the configuration 
$N(t)=(\sqrt{1-n_3^2},0,n_3)$ has energy $F_\infty(N)\le F_\infty(n)$.  Indeed, we calculate
$$   |\dot N|^2 = {\dot n_3^2\over 1- n_3^2}, \qquad 
    g(N)=g(n)=\sqrt{\frac32}(1-n_3^2),  $$
and
$$   |\dot n|^2 - |\dot N|^2 = 
    {(1-n_3^2)(\dot n_1^2 + \dot n_2^2)- [n_3\dot n_3]^2\over
                        1-n_3^2}=
                        {(n_1^2+n_2^2)(\dot n_1^2+\dot n_2^2)-(n_1\dot n_1 + n_2\dot n_2)^2\over 1-n_3^2} \ge 0,
$$     
by the Cauchy-Schwartz inequality.  Thus, it is sufficient to consider $\varphi=0$, $n=N$, and 
$$  F_\infty(n)=\int_0^\infty \left[ {\dot n_3^2\over (1-n_3^2)} + \sqrt{\frac32}(1-n_3)^2\right] dt.
$$
Moreover, the curve $\gamma$ traced by $n(t)$ follows a meridian on the sphere.

Following \cite{St91}, we note that
\beqn\label{MM}  F_\infty(n) \ge \int_0^\infty 2\sqrt{g(n)}|\dot n|\, dt
      = \int_\gamma \kappa\sqrt{1-n_3^2}\, ds,  
\eeqn
where $\gamma$ is the curve traced out by $n(t)$, $\kappa=\sqrt[4]{24}$, and the integral is with respect to arclength $ds$ on $\gamma$.  Equality holds when
$|\dot n|=\sqrt{g(n)}$, that is, 
$$  {|\dot n_3|\over 1-n_3^2} = {\kappa\over 2},
$$
which may be integrated to give an explicit formula for the heteroclinic,
\begin{gather}\label{eq:hetero}
n_3(t,\theta,0)= {A(\theta)-e^{-\kappa t}\over A(\theta)+ e^{-\kappa t}}, 
\quad n_1=\sqrt{1-n_3^2}, \\
\text{with} \quad A(\theta)={1+\cos\theta\over 1-\cos\theta}. \nonumber
\end{gather}
Clearly, $n$ is smooth with respect to both $t$ and $\theta\in (0,\pi]$, and a simple calculation shows that ${\partial n\over\partial\theta}(t,0)=0$, and so it is smooth for all $(t,\theta)$.  The exponential decay also follows from direct calculation.  Finally, to evaluate the energy at a minimizer, recall that in \eqref{MM} equality is achieved at a minimizer, and so
$$  F_\infty(n)=\int_{\gamma}\kappa\sqrt{1-n_3^2} \ ds
    =\kappa \int_0^\theta \sin\theta\, d\theta=\kappa(1-\cos\theta).  $$
\end{proof}

\begin{remark}\label{rem1}
It is easy to see that the minimizer $n(t,\theta,\varphi)$ of $d_\infty^-(\theta,\varphi)$ has energy $F_\infty(n)=\kappa(1+\cos\theta)$, and so 
$D_\infty(Q_b(\theta,\varphi))=\kappa(1-|\cos\theta|)$.
\end{remark}

\medskip

We are now ready to prove our upper bound proposition.

\begin{proof}[Proof of Proposition~\ref{p:upII}]
We construct an axially symmetric  map $Q_\xi$ of the form
\begin{equation}\label{Qeta}
Q_\xi(r,\theta,\varphi)
= {}^tR_\varphi \Qb\left(r,\theta\right) R_\varphi,
\end{equation}
where $R_\varphi$ is (as before) the rotation of angle $\varphi$ about the axis $\mathbf e_3$.
As above, in spherical coordinates we decompose the gradient as:
\begin{align}\label{eq:energydensity}
\abs{\nabla Q_\xi}^2 &=
\abs{\frac{\partial\Qb}{\partial r}}^2 + \frac{1}{r^2}\abs{\frac{\partial\Qb}{\partial\theta}}^2+\frac{1}{r^2\sin^2\theta}\Xi[\Qb],\\
\nonumber
\text{with}\quad
\Xi[\Qb]&= \abs{ \partial_\varphi[ {}^t R_\varphi \Qb R_\varphi]}^2.
\end{align}
As the energy will be the same in each vertical cross-section $\{\varphi=\text{constant}\}$ it will be convenient to define a two-dimensional energy,
\begin{multline*} \Et(\Qb;U):= \iint_U  \biggl[ 
   \frac 12\abs{\frac{\partial\Qb}{\partial r}}^2 + \frac{1}{2r^2}\abs{\frac{\partial\Qb}{\partial\theta}}^2+\frac{1}{2r^2\sin^2\theta}\Xi[\Qb] 
   \\ + {1\over\xi^2}f(\Qb)
       +{1\over\eta^2}g(\Qb)\biggr] r^2\sin\theta\, dr\, d\theta,
\end{multline*}
for $U\subset\Omega_0:=\lbrace (r,\theta)\colon r>1,\, 0 < \theta < \pi\rbrace$.
We construct $\Qb$ in the upper half $\Om_0^+:=\{(r,\theta): \ r>1, \ 0\le\theta<{\pi\over 2}\}$ of the cross-section $\{\varphi=0\}$, and define its value in the lower cross-section $\theta\in ({\pi\over 2},\pi]$ by reflection,
$$  \Qb(r,\theta) = T\Qb(r,\pi-\theta)T^t, \quad
\text{where} \ T(x,y,z)=(x,y,-z).  $$
Moreover, we divide the region $\{(r,\theta): \ r>1, \ 0\le\theta<{\pi\over 2}\}$ into three subregions, and define $\Qb$ as a smooth map in each, continuous across the common boundaries.

\bigskip

\begin{figure}[h]
\begin{center}
\begin{tikzpicture}[scale=1.5]

\draw (1,0) arc (0:90:1);
\draw (1,0) -- (5,0);
\draw (0,1) -- (0,4);
\draw (30:1)-- ++(30:4);
\draw (2.5,0) arc (0:30:2.5);
\draw[<->] (.9,0) arc (0:30:.9) node [below left] {$\eta$};
\draw[<->] (1,-.1) -- (2.5,-.1);
\draw (1.7,-.1) node [below] {$2\eta$};
\draw (1.5,2) node {$\Omega_1$};
\draw (3.5,1) node {$\Omega_2$};
\draw (1.7,.5) node {$\Omega_3^+$};

\end{tikzpicture}
\end{center}

\caption{The three subregions of $\Omega_0^+$ used in the proof of Proposition~\ref{p:upII}}
\end{figure}
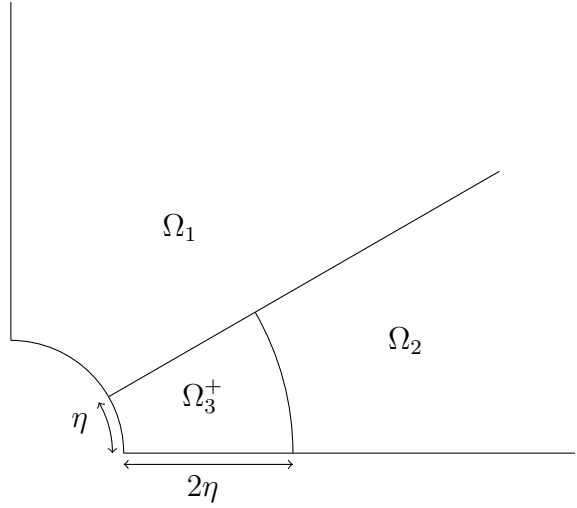

\medskip

\noindent {\sl Region 1:} $\Om_1=\{(r,\theta): \ r>1, 0\le\theta\le {\pi\over 2}-\eta\}$.  In this region, $\Qb$ will be uniaxial,
$\Qb = \nb \otimes \nb-\frac13 I$, for a director field $\nb\in\Stwo$.  Specifically, let \beqn \label{eq:ndef}
n(t,\theta)= \begin{bmatrix} n_1(t,\theta) \\ 0 \\ n_3(t,\theta)\end{bmatrix}
   = \begin{bmatrix} \sqrt{1-n^2_3(t,\theta)} \\ 0 \\ n_3(t,\theta)\end{bmatrix}.
\eeqn 
denote the minimizing geodesic which attains the distance 
$D_\infty(Q_b(\theta,0))$, and whose explicit formula is given in \eqref{eq:hetero}.  Then, for $(r,\theta)\in\Om_1$ and $\rt = (r-1)/\eta$, we set
$$  \nb(r,\theta):= n(\rt,\theta)=n\left({r-1\over\eta},\theta\right), \quad 
\Qb(r,\theta):= \nb \otimes \nb-\frac13 I.  $$

Using the above expression of the energy density \eqref{eq:energydensity} we derive
\begin{align*}  \frac12 |\nabla Q_\xi|^2 &=
 {1\over \eta^2} \left|{\partial n\over\partial\rt}\right|^2
      + {1\over r^2}\left|{\partial n\over\partial\theta}\right|^2
      + {1\over r^2\sin^2\theta} |n_1|^2\\
      &= {1\over 1-n_3^2}\left[
     {1\over\eta^2} \left|{\partial n_3\over\partial\rt}\right|^2 +
      {1\over r^2} \left|{\partial n_3\over\partial\theta}\right|^2
      \right] + {1\over r^2\sin^2\theta} (1-n_3^2),
\end{align*}
as the $\varphi$ derivative term simplifies to $\Xi(\Qb)= 2|\nb_1|^2$.  As $\Qb$ is uniaxial, $f(\Qb)\equiv 0$, and $g(\Qb)=\sqrt{\frac32}(1-\nb_3^2)$. The energy in $\Om_1$ then becomes, after the change of variable $r=1+\eta\rt$,
\begin{align} \nnn
\eta \Et(\Qb;\Om_1) &= \int_0^{\frac{\pi}2-\eta}\int_0^\infty \biggl[ 
 \left|{\partial n\over\partial\rt}\right|^2 +  g(n) \\
 \nnn
    &\qquad
      + {\eta^2\over (1+\eta \rt)}\left|{\partial n\over\partial\theta}\right|^2
      + {\eta^2\over (1+\eta \rt)^2\sin^2\theta} |n_1|^2
\biggr] (1+\eta \rt)^2\, \sin\theta\, d\rt\, d\theta \\
\nnn
&= \int_0^{\frac{\pi}2-\eta} F_\infty(n(\cdot,\theta,0))\, \sin\theta\, d\theta + O(\eta)\\
\label{eq:Om1}
&\le \int_0^{\frac{\pi}2} F_\infty(n(\cdot,\theta,0))\,\sin\theta\, d\theta + O(\eta),
\end{align}
since by the exponential decay estimates of Lemma~\ref{L:geod}, each of the remaining integrals converges, and carries at least one factor of $\eta$.

\medskip

\noindent {\sl Region 2:} $\Om_2=\{(r,\theta): \ r\ge 1+2\eta, {\pi\over 2}-\eta \le\theta\le {\pi\over 2}\}$.  By the exponential decay of $\nb$ to $\mathbf e_3$, the value of $\Qb$ on the ray $r\ge 1+2\eta$, $\theta={\pi\over 2}-\eta$ is already close to $Q_\infty$, so here we interpolate between the two in this sector.  Define $\Phi_\eta^+(\rt)$ to be the spherical angle associated to the heteroclinic 
$n^+(\rt):=n(\rt,{\pi\over 2}-\eta)$, that is,
$$ n^+(\rt)= n(\rt,{\pi\over 2}-\eta) = \left(\sin\Phi_\eta^+(\rt), \ 0, \ \cos\Phi_\eta^+(\rt)\right), \quad \rt\ge 0. $$
We note for later use that the exponential decay of $n$ to $\mathbf e_3$ implies that the angle $\Phi_\eta^+(\rt)$ also has exponential decay to zero as $\rt\to\infty$.

We extend $\Qb$ to $\Om_2$ uniaxially by interpolating this angle:  define
$$  \Phi(\rt,\theta):= \Phi_\eta^+(\rt)\, \chi(\theta),\qquad \text{with}\quad 
  \chi(\theta):={\frac{\pi}2 -\theta\over \eta},\quad
      {\pi\over 2}-\eta\le\theta \le {\pi\over 2}. $$
Then, for $r\ge 1+2\eta$ and ${\pi\over 2}-\eta<\theta \le {\pi\over 2}$ we set
$$ \hat n(t,\theta):= (\sin\Phi(\rt,\theta), 0 , \cos\Phi(\rt,\theta)) \ \text{and} \ 
    \Qb(r,\theta):= \hat n(\rt,\theta)\otimes \hat n(\rt,\theta) - \frac13 I,
$$
where (as usual)  $r=1+\eta \rt$.  

To evaluate the energy in this sector we use
\begin{align*} \frac12 |\nabla Q_\xi|^2 &
= \left|{\partial\Phi\over\partial r}\right|^2
        + {1\over r^2}\left|{\partial\Phi\over\partial \theta}\right|^2
          + {1\over r^2\sin^2\theta}\sin^2\Phi  \\
     &=  \chi^2(\theta)\left|{\partial\Phi_\eta^+\over\partial r}\right|^2
           + {1\over \eta^2 r^2}\left|\Phi_\eta^+\right|^2
             + {1\over r^2\sin^2\theta}\sin^2[\Phi_\eta^+\chi(\theta)] \\
     &\le \left|{\partial\Phi_\eta^+\over\partial r}\right|^2
           + {1\over \eta^2 r^2}\left|\Phi_\eta^+\right|^2
             + {1\over r^2\sin^2\theta}\sin^2\Phi_\eta^+ \\
      &= \left|{\partial n^+\over\partial r}\right|^2
           + {1\over r^2\sin^2\theta}[n^+_1]^2
            + {1\over \eta^2 r^2}\left|\Phi_\eta^+\right|^2.
\end{align*}

We then calculate the energy, recalling that $\Qb$ is uniaxial, and so $f(\Qb)=0$ and 
$$ g(\Qb)=\sqrt{\frac32}\sin^2\Phi\le\sqrt{\frac32}\sin^2\Phi_\eta^+ 
   =g\left(n^+\right).  $$
Changing variables from $r=1+\eta \rt$, since each term in the integrand is bounded by a decaying exponential in $\rt$, we have the estimate:
\begin{align*}
\eta\Et(\Qb;\Om_2)
& \le  \int_{\frac{\pi}2-\eta}^{{\pi\over 2}}\int_2^\infty\biggl[
   \left|{\partial n^+\over\partial \rt}\right|^2
   + g\left(\nb^+(r)\right)  \\
       &\qquad \quad   + {1\over (1+\eta\rt)^2\sin^2\theta}[n_1^+(r)]^2
            + {1\over (1+\eta\rt)^2}\left|\Phi_\eta^+\right|^2 \biggr] (1+\eta \rt)^2\sin\theta\, d\rt\, d\theta \\
                        &\le O(\eta).
\end{align*}

Note that when $\theta=\frac{\pi}2$, $n_3=\mathbf e_3$.  When reflecting to the lower half of the cross-section this will create a discontinuity in the director field, but will be invisible in the tensor $\Qb$, which will take the value $Q_\infty$ continuously across the equatorial plane.

\medskip

\noindent {\sl Region 3:} $\Om_3^+=\{(r,\theta): \ 1<r<1+2\eta, {\pi\over 2}-\eta \le\theta\le {\pi\over 2}\}$.  Unlike the other regions, here our test configuration will not be uniaxial; it is here that we imagine that the Saturn ring defect will occur.

It will be convenient to construct $\Qb$ in the symmetric domain obtained by reflection across the equatorial plane, $\Om_3:=\{(r,\theta): \ 1<r<1+2\eta, {\pi\over 2}-\eta \le\theta\le {\pi\over 2}+\eta\}$.  We note that by the previous steps (and the definition of $\Qb$ by reflection to the lower hemisphere,) the values of $\Qb$ have already been determined on $\partial\Om_3$; in particular, $\Qb|_{\partial\Om_3}$ is uniaxial, with director which carries a degree of $-\frac12$.

Consider the square domain $\Omt_3=\{-1<s<1, \ -1<\tau<1\}$, which is obtained from $\Om_3$ via the change of variables
\beqn\label{eq:st}
 r=1+\eta(s+1),\quad \theta={\pi\over 2}-\eta \tau.  
\eeqn
   Note that here we are considering $(s,\tau)$ as Cartesian coordinates, with Jacobian $dr\, d\theta=\eta^2 ds\, d\tau$.  
We will define $\Qb(r,\theta)=\Qt(s,\tau)$ for $(s,\tau)\in\Omt_3$, with $\Qt$ the solution of an appropriate boundary value problem.  
The energy in $\Om_3$ transforms as,
\begin{align}\label{energy3}
\eta\Et(\Qb;\Om_3) &= \eta\hat E(\Qt;\Omt_3):=\eta \iint_{\Omt_3} \biggl[
   \left|{\partial \Qt\over\partial s}\right|^2
   + {1\over r(s)^2} \left|{\partial \Qt\over\partial \tau}\right|^2  \\ \nnn
     &\qquad + {1\over r(s)^2\sin^2 \theta(\tau)} |\Xi(\Qt)|^2 
      + {1\over \e^2}f(\Qt) + g(\Qt)\biggr] r(s)^2\, \sin\theta(\tau)\, ds\, d\tau,
\end{align}
with $r(s),\theta(t)$ as in \eqref{eq:st}, and $\e:=\xi/\eta\to 0$.
The boundary conditions induced from $\Qb|_{\partial\Om_3}$, given in terms of the director field, are:
\begin{itemize}
\item $n_\eta^+(s+1)=n(s+1,{\pi\over 2}-\eta)$, for $s\in [-1,1]$, $\tau=1$;
\item its reflection, $Tn_\eta^+(s+1)=n(s+1,{\pi\over 2}+\eta)$, for $s\in [-1,1]$, $\tau=-1$;
\item the rescaled homeotropic condition, $(\cos(\tau\eta),0,\sin(\tau\eta))$, for $s=-1$, $\tau\in [-1,1]$;
\item  the interpolated field from Region 2, 
  $(\sin(\tau\Phi_\eta^+(2)), 0, \text{sgn}(\tau)\cos(\tau\Phi_\eta^+(2))$, for $s=1$, $-1\le \tau\le 1$, which is discontinuous but well-defined as a Q-tensor.
\end{itemize}
Moreover each component converges in $C^1$ as $\eta\to 0$, and the boundary conditions determine a degree $-\frac12$ map on $\partial\Omt_3$. 

Introducing polar coordinates $(\rho,\alpha)$ in $\Omt_3$, we parametrize the square $\partial\Omt_3$ with respect to the polar angle, $\rho=\gamma(\alpha)$, $0\le \alpha<2\pi$.  The boundary conditions given above may then be described in terms of this parametrization of $\partial\Omt_3$ via a piecewise smooth phase $\Psi_\eta$,
in the form
$$  \hat n_\eta|_{\partial\Omt_3} 
  = \bigl\{ (\sin\Psi_\eta(\alpha),0,\cos\Psi_\eta(\alpha)): \ 
  \text{on $\rho=\gamma(\alpha)$, $0\le\alpha<2\pi$}\bigr\},
$$
which defines a continuous and piecewise smooth uniaxial tensor $\hat Q_\eta$.
In a similar way we define $\Psi_0$, $\hat n_0$, $\hat Q_0$ corresponding to the $\eta\to 0$ limits of the boundary value components, parametrized by the polar angle $\alpha$.

We first define $\Qt$ in the square annulus $\Omt_3\setminus\overline{\Omt_{3/2}}$, where $\Omt_{3/2}=[-\frac12,\frac12]\times [-\frac12,\frac12]$ and is parametrized in polar coordinates by $\rho=\frac12 \gamma(\alpha)$.  As in $\Om_2$, we extend $\Qt$ as a uniaxial tensor by interpolating the phase angle associated to its director, but here we interpolate along radii,
$$  \hat\Psi_\eta(\rho,\alpha):= 
     {2\rho-\gamma(\alpha)\over \gamma(\alpha)}\Psi_\eta(\alpha)
     + {2\gamma(\alpha)-2\rho\over \gamma(\alpha)}\Psi_0(\alpha),
$$
with $\hat n_\eta(\rho,\alpha):=(\sin\hat\Psi_\eta,0,\cos\hat\Psi_\eta)$ and $\Qt:=\hat n_\eta\otimes \hat n_\eta-\frac13 I$.  Since $\Qt$ is piecewise smooth and $\Psi_\eta\to\Psi_0$ in $C^1$ on each edge of the square, by inserting in \eqref{energy3} we obtain 
$$ \eta\hat E(\Qt;\Omt_3\setminus\Omt_{3/2}) \le O(\eta).  $$

It remains to define $\Qt$ in the smaller square $\Omt_{3/2}$.  Here the boundary data is uniaxial and $\eta$-independent, given by the phase angle $\Psi_0(\alpha)$.  Here we follow \cite{baumanparkphillips12} and introduce a complex order parameter $u=u_1 + i u_2$ to parametrize a $Q$-tensor of the form
\beqn\label{Qu}
  \Qt = \begin{bmatrix}
    \frac 1{\sqrt 6}u_1+\frac13 &  0 &  \frac 1{\sqrt 6} u_2 \\
    0 & -\frac23 & 0 \\
    \frac 1{\sqrt 6} u_2 & 0 & \frac13 - \frac 1{\sqrt 6} u_1
\end{bmatrix}.
\eeqn
The value $\hat n_0$ given on $\partial\Omt_{3/2}$ determines a boundary value for $u$; although $\hat n_0$ jumps (from $\pm \mathbf e_3$) at $(s,\tau)=(\frac12,0)$, the boundary value for $u$ will be well-defined there, and continuous and piecewise smooth on all of $\partial\Omt_{3/2}$, with degree $-1$.  We then minimize the Ginzburg--Landau energy with this given boundary condition to obtain $u:\ \Omt_{3/2}\to\mathbb{C}$ with
$$  \int_{\Omt_{3/2}} \left[\frac 16 |\nabla u|^2 + {1\over 2\e^2}(1-|u|^2)^2\right]
            ds\, d\tau \le \frac\pi 3|\ln \e| + C,  $$
with constant $C$ independent of $\e$. 

Using \eqref{Qu} to define $\Qt$ from $u$, we may thus estimate
$$   \int_{\Omt_{3/2}} \left[ \frac 12 \left|{\partial\Qt\over\partial s}\right|^2  
                            +  \frac 12 \left|{\partial\Qt\over\partial \tau}\right|^2
                            + {1\over \e^2} f(\Qt)\right] ds\, d\tau \le 
                                \frac\pi 3|\ln \e| + C.
$$                                
Comparing with \eqref{energy3}, we note that $r(s),\sin\theta(t)\to 1$ uniformly on $\Omt_3$, and hence we may conclude that 
$$  \eta \hat E(\Qt; \Omt_{3/2}) \le \frac\pi 3 \eta|\ln \e| + O(\eta).  $$

In conclusion, the only nontrivial contribution to the energy at order $\frac{1}{\eta}$ comes from Region 1, and, extending the definition of $\Qb$ by reflection to the entire cross-section $\Omega_0=\lbrace r>1,\, 0<\theta<\pi\rbrace$, we obtain the desired upper bound,
$$   \widetilde E(\overline Q_\xi;\Om_0)\le
\frac 1\eta  
 \int_0^{\pi} F_\infty(n(\cdot,\theta,0))\, \sin\theta\, d\theta + \frac \pi 3 |\ln\e| + O(1).
$$
Defining $Q_\eta$ via \eqref{Qeta}, we complete the proof of the proposition.
\end{proof}

\section{Proving Theorem~\ref{t:asympt} and Proposition \ref{p:enercomp}}\
\label{s:proof}

\begin{proof}[Proof of Theorem~\ref{t:asympt}]
For any measurable $U\subset\mathbb S^2$ we have by Proposition~\ref{p:low}
\begin{equation}\label{eq:low}
E_\xi(Q_\xi;\Omega\cap\mathcal C(U))\geq \frac{1}{\eta}\int_U D_\lambda(Q_b(\omega))\,d\mathcal H^2(\omega)+o\left(\frac 1\eta\right).
\end{equation}
On the other hand, using \eqref{eq:low} and the upper bound proved in \S~\ref{s:up} we obtain
\begin{align*}
E_\xi(Q_\xi;\Omega\cap\mathcal C(U))&
=E_\xi(Q_\xi;\Omega)-E_\xi(Q_\xi;\Omega\cap\mathcal C(\mathbb S^2\setminus  U))\\
&\leq 
\int_{\mathbb S^2} D_\lambda(Q_b(\omega))\,d\mathcal H^2(\omega)
- \int_{\mathbb S^2\setminus  U} D_\lambda(Q_b(\omega))\,d\mathcal H^2(\omega) + o\left(\frac 1\eta\right) \\
& = \int_{U} D_\lambda(Q_b(\omega))\,d\mathcal H^2(\omega) + o\left(\frac 1\eta\right).
\end{align*}
\end{proof}

\medskip

\begin{proof}[Proof of Proposition \ref{p:enercomp}]
Let $Q= n\otimes n -\frac13 I$ with $n\in\mathcal{N}$, for which $E_\xi(Q)<\infty$.  Then
$$  \eta E_\xi(Q) = \int_\Om\left[\eta |\nabla n|^2 + {1\over \eta} g(n)\right] dx<\infty.
$$
In particular, by Fubini's theorem, for almost every $\omega\in\Stwo$ and $\eta,\xi$ fixed, we have
$$  \int_1^\infty \left[\eta \left|{\partial n\over\partial r}\right|^2 + {1\over \eta} g(n)\right] r^2 \, dr<\infty,
$$
and hence on almost every ray, $n(r,\omega)\to \pm \mathbf e_3$ as $r\to\infty$.  Again by Fubini's theorem, $n(r,\cdot)\in H^1(\Stwo;\Stwo)$ for almost every $r>1$, and so either $n(r,\omega)\to \mathbf e_3$ for almost all $\omega\in\Stwo$ or $n(r,\omega)\to -\mathbf e_3$ for almost all $\omega\in\Stwo$.  Without loss, we assume the former, $n(r,\omega)\to \mathbf e_3$ a.e.  In particular, after the familiar change of variables $r=1+\eta t$, $\hat n(t,\omega):=n(r,\omega)$ is an admissible function for the minimization problem $d_\infty^+$ for a.e. $\omega$, and so,
\begin{align*}
\eta E_\xi(Q) &= \int_{\Stwo}\int_0^\infty \left[ 
   \left|{\partial \hat n\over\partial r}\right|^2 + \eta |\nabla_\omega \hat n|^2 +  g(n)\right] (1+\eta t)^2\, dt\, d\mathcal{H}^2(\omega) \\
   &\ge \int_{\Stwo} F_\infty(\hat n(\cdot,\omega)) 
       \, d\mathcal{H}^2(\omega) \\
 &\ge   \int_{\Stwo} d_\infty^+(\omega) \, d\mathcal{H}^2(\omega) \\
 &=  \int_0^{2\pi}\int_0^\pi \kappa (1-\cos\theta)\, \sin\theta d\theta\, d\varphi = 8\pi\kappa,
\end{align*}
by Lemma~\ref{L:geod}.  

 On the other hand, we note that $D_\lambda\le D_\infty$ for any $\lambda\in (0,\infty]$, 
 since the domain of $F_\lambda$ contains the domain of $F_\infty$, and on the latter both functionals coincide. Thus, for any $\lambda\in (0,\infty]$, by \eqref{liminfty},
$$  \lim_{\xi\to 0\atop {\eta\over\xi}\to\lambda} 
      \left(\min_{\mathcal{H}}\eta E_\xi\right)
   \le \lim_{\xi\to 0\atop {\eta\over\xi}\to\infty} 
       \left(\min_{\mathcal{H}}\eta E_\xi\right)
   =2\pi\kappa, $$
and the proposition follows.

\end{proof}

\bibliographystyle{acm}

\end{document}